\documentclass{article}

\usepackage{arxiv}

\usepackage[utf8]{inputenc} 
\usepackage[T1]{fontenc}    
\usepackage{hyperref}       
\usepackage{url}            
\usepackage{booktabs}       
\usepackage{amsfonts}       
\usepackage{nicefrac}       
\usepackage{microtype}      
\usepackage{wrapfig}
\usepackage{subcaption}
\usepackage[export]{adjustbox}
\usepackage{graphicx}
\usepackage{amsmath}
\usepackage{amsthm}
\usepackage{cancel}
\usepackage{float}
\newtheorem{thm}{Theorem}[section]
\newtheorem{conj}[thm]{Conjecture}
\newtheorem{lem}[thm]{Lemma}
\newtheorem{prop}[thm]{Proposition}
\newtheorem{cor}[thm]{Corollary}
\theoremstyle{definition}
\newtheorem{exmp}{Example}[section]
\theoremstyle{remark}
\newtheorem{rem}{Remark}
\theoremstyle{definition}
\newtheorem{defn}{Definition}[section]

\title{Variations on Ramsey numbers and minimum numbers of monochromatic triangles in line $2$-colorings of configurations}

\author{
  Jamie Bishop, Rebekah Kuss, Benjamin Peet Ph.D.\\
  Department of Mathematics\\
  St. Martin's University\\
  Lacey, WA 98503 \\
  \texttt{bpeet@stmartin.edu} \\
}

\begin{document}
\maketitle

\begin{abstract}
This paper begins by exploring some old and new results about Ramsey numbers and minimum numbers of monochromatic triangles in $2$-colorings of complete graphs, both in the disjoint and non-disjoint cases. We then extend the theory, by defining line $2$-colorings of configurations of points and lines and considering the minimum number of non-disjoint monochromatic triangles. We compute specific examples for notable symmetric $v_{3}$ configurations before considering a general result regarding the addition or connected sum of configurations through incidence switches. The paper finishes by considering the maximal number of mutually intersecting lines and how this relates to the minimum number of triangles given a line $2$-coloring of a symmetric $v_{3}$ configuration.
\end{abstract}


\section*{Acknowledgement}

We would like to thank Misha Lavrov for his very helpful contribution to Theorem 3.7 and connecting us with the relevant background literature for this paper, as well as the anonymous reviewer for their comprehensive report which greatly improved the paper.

\section{Introduction}

Ramsey numbers came from the mathematician, philosopher, and economist Frank Plumpton Ramsey. In the years since Ramsey, there have been many developments of Ramsey numbers and this paper seeks to first give an exposition of some different versions of Ramsey numbers specifically for triangles. In \cite{goodman1959sets}, the minimum number of monochromatic triangles given an edge $2$-coloring was considered. We continue the theory on to an adaptation of Ramsey theory for configurations of points and lines considering the minimum number of monochromatic triangles given a $2$-coloring.

Configurations of points and lines are a classical topic, but the literature has not brought these two subjects together. We discuss where graph theory and the theory of configurations diverge and how this relates to Ramsey theory. This paper gives a number of specific results regarding the number of non-disjoint monochromatic triangles in some notable configurations, as well as some results about the addition or connected sum (using incidence switches) of configurations and the maximum number of mutually intersecting lines.

\section{Preliminary definitions and examples}

\subsection{Graph theory} We include the following preliminary definitions for completeness. These can be found in most introductory texts on graph theory. We referenced \cite{harary1969graph}.

A \textit{graph} consists of finite sets $V$ and $X$. $V$ is the set of vertices and $X$ is the set of unordered pairs of distinct vertices of $V$. That is, the edges.

A \textit{subgraph} $G^{\prime}$ of a graph $G$ consists of a subset of vertices $V^{\prime}$ of $V$ and a subset $X^{\prime}$ of $X$ that contains only pairs of vertices from $V^{\prime}$.

An \textit{r-coloring} of a set $S$ is a map $x:s\rightarrow (r)$ for $s\in S$, $(x(s))$ is called the color of $s$. This paper will deal only with $2$-colorings or bi-colorings and as convention dictates consider these as red-blue edge colorings, hence $x(s)=\text{red}$ or $x(s)=\text{blue}$.

A subgraph of a colored graph is \textit{monochromatic} if all its edges have the same color.

A \textit{complete graph} $K_p$ is a graph of order $p$ with all possible pairs of vertices connected by an edge.

Two subgraphs of a given graph are \textit{vertex disjoint} if no vertices are shared. Two subgraphs of a given graph are \textit{edge disjoint} if no edges are shared. Note that vertex disjoint implies edge disjoint but not the converse. For convenience here we deal only with vertex disjoint and simply say \textit{disjoint}.

A sequence of distinct vertices $x_1 x_2 \ldots x_n$ is called a \textit{path} if $x_i x_i+1$ is an edge for all \textit{i} such that $1\leq i \leq n-1$. When $x_1 x_n$ is also an edge, the sequence $x
_1 x_2 \ldots x_n x_1$ is called \textit{cycle}.

\subsection{Ramsey theory}

The \textit{Ramsey number} $r(G_1,G_2)$ is defined to be the least number of vertices $p$ to guarantee that a red-blue colored ($2$-colored) complete graph on $p$ vertices contains a red subgraph isomorphic to $G_1$ or a complete blue subgraph isomorphic to $G_2$.

We then define the \textit{multiple Ramsey numbers} \cite{burr1975ramsey} as $r(kG_1,nG_2)$ being the least number of vertices $p$ to guarantee that a red-blue colored ($2$-colored) complete graph on $p$ vertices contains $k$ disjoint red subgraphs isomorphic to $G_1$ or $l$ disjoint blue subgraphs isomorphic to $G_2$. Here $kG_1$ represents $k$ disjoint copies of $G_1$.

\subsection{Configurations}

An incidence geometry is given by a pair $(\mathcal{P},\mathcal{L})$ where $\mathcal{P}=\{p_{1}, \ldots , p_{n}\} $; $\mathcal{L}=\{l_{1}, \ldots ,l_{m}\}$; each $l_{i}\subset \mathcal{P}$; and for any pair $p_{i_{1}},p_{i_{2}}$ there is at most one line $l_{j}$ that contains both elements. Naturally, the elements of $\mathcal{P}$ are known as points and the elements of $\mathcal{L}$ are known as lines.

An incidence geometry is further known as a configuration if each point is incident with the same number of lines as any other (denoted $s$); each line is incident with the same number of points as any other (denoted $k$); and there are at least three points.

Configurations are here considered in a purely combinatorial sense.

When a configuration has the parameters $(n,m,s,k)$, we term it a $(n_s,m_k)$-configuration. If $n=m$ (and consequently $s=k$ - see the foundational text on configurations of \cite{grunbaum2009configurations}) we call the configuration \textit{symmetric} and usually denote it a $v_k$ configuration according to the commonly used notation in papers such as \cite{betten2000counting} and \cite{boben2007irreducible}.

A \textit{Menger graph} is a representation of a configuration as an undirected graph where the points are shown as the vertices and the lines are made up by a collection of edges. See  \cite{van1947topology} and \cite{coxeter1947configurations}.

\section{Preliminary results}

The following is Ramsey's theorem. We used the version from \cite{do2019party}: 

\begin{thm}
For every pair of positive integers $m$ and $n$, the value of $r(K_m,K_n)$ is finite. In other words, there is a positive integer $N$ such that if the edges of $K_{N}$ are coloured red or blue, then there exists a red subgraph isomorphic to $K_{m}$ or a blue subgraph isomorphic to $K_{n}$.
\end{thm}

We now give a Lemma possible due to \cite{kery} that we will use in our new results:

\begin{lem}
Any red-blue coloring of $K_{5}$ that does not contain a red triangle or a blue triangle must be composed of a blue $5$-cycle and a red $5$-cycle.
\end{lem}
\begin{proof} Take a 2-coloring of $K_5$ with no monochromatic triangle. Note that no vertex is incident with three or more edges of the same color - if $ab$, $ac$, and $ad$ are all red, the subgraph induced by ${b, c, d}$ cannot contain any red edges without creating a red triangle. Furthermore, if all of its edges are blue, then it creates a blue triangle. Hence, every vertex must be incident with exactly two red edges and two blue edges. 

Without loss of generality, label the vertices $a, b, c, d, e$. Let $ab$ and $ac$ be red and $ad$ and $ae$ be blue. So $bc$ is also blue, $de$ is red, and $b$ must join to exactly one of $d$ and $e$ as a red edge.  Again, without loss of generality, we can assume that $bd$ is red and $be$ is blue. This forces $cd$ to be blue and $ce$ to be red.  This gives the red cycle $abdeca$ and the blue cycle $adcbea$. 
\end{proof}

We now continue the theory by giving two more versions of Ramsey numbers. These are not new, just simplified versions of those seen in the literature where notation and terminology varies.

First of all we disregard which color the subgraph is. This leads us to:

\begin{defn}

Let $r(G_1,G_2:k)$ be the least number of vertices to guarantee a total of at least $k$ disjoint subgraphs that are either red and isomorphic to $G_1$ or blue and isomorphic to $G_2$.

\end{defn}

Note that in this paper we consider only when $G_1,G_2$ are the same complete graph, but if $m> n$ the specific statement of "$k$ disjoint subgraphs that are either red and isomorphic to $K_m$ or blue and isomorphic to $K_n$" as opposed to "$k$ disjoint monochromatic subgraphs that are isomorphic to $K_m$ or $K_n$" is required as a red $K_m$ would contain many copies of a red $K_n$.

We also then consider removing the disjoint condition.

\begin{defn}
We define $r^*(kG_1,lG_2)$ to be the least number of vertices to guarantee $k$ red subgraphs isomoprhic to $G_1$ or $l$ blue subgraphs isomorphic to $G_2$.
\end{defn}

The final definition is a combination of Definition 3.1 and 3.2:

\begin{defn}
We define $r^*(G_1,G_2:k)$ to be the least number of vertices to guarantee a total of at least $k$ subgraphs that are either red and isomorphic to $G_1$ or blue and isomorphic to $G_1$.
\end{defn}

We make the observation here that these numbers are closely related to Ramsey multiplicities. The general definition of a Ramsey multiplicity $R(G_1,G_2)$ is given by the smallest possible total number of subgraphs that are red and isomorphic to $G_1$ or blue and isomorphic to $G_2$ among all $2$-colorings of $K_{r(G_1,G_2)}$. Note the use of lower case $r$ for Ramsey number and upper case $R$ for Ramsey multiplicity. 

We give a quick example for exposition:

\begin{exmp}

We will see in Proposition 3.4 that $r^{*}(K_{3},K_{3}:2)=6$ and it is a first result of Ramsey theory that $r(K_{3},K_{3})=6$. Simply put, 6 vertices are required to guarantee a monochromatic triangle in any $2$-coloring and in fact this guarantees 2 of them.

Now $R(K_{3},K_{3})$ represents the least possible number of monochromatic triangles in $K_{r(K_3,K_3)}=K_6$. Clearly, $R(K_{3},K_{3})=2$. 
\end{exmp}

We take this example to prove the following general result:

\begin{prop}
    $r^{*}(G_1,G_2:R(G_1,G_2))= r(G_1,G_2)$
\end{prop}

\begin{proof}

    Any $2$-coloring of $K_{r(G_1,G_2)}$ contains at a minimum $R(G_1,G_2)$ subgraphs that are either red and isomorphic to $G_1$ or blue and isomorphic to $G_2$.

    Now $r^{*}(G_1,G_2:R(G_1,G_2))$ is the least number of vertices $p$ so that any $2$-coloring of $K_p$ contains at a minimum $R(G_1,G_2)$ subgraphs that are either red and isomorphic to $G_1$ or blue and isomorphic to $G_2$.

    Hence $r^{*}(G_1,G_2:R(G_1,G_2))\leq r(G_1,G_2)$.

    Now, by definition there is a $2$-coloring of $K_{r(G_1,G_2)-1}$ that contains no subgraphs that are either red and isomorphic to $G_1$ or blue and isomorphic to $G_2$. Hence it certainly cannot have at least $R(G_1,G_2)$ of them.

    So $r^{*}(G_1,G_2:R(G_1,G_2))>r(G_1,G_2)-1$ and the equality follows.
\end{proof}

These Ramsey multiplicities have been considered in many places, in particular \cite{burr1980ramsey}, \cite{conlon2007ramsey}, \cite{harary1974generalized}, and \cite{jacobson1980note}.

We now give a result without proof from \cite{burr1975ramsey}:

\begin{thm}
$r(kK_3,lK_3) =3k+2l$ for all $k\geq l\geq 1$ and $k\geq 2$.
\end{thm}

\begin{rem}
We give here inequalities that are useful in relating the various Ramsey numbers:

$$r(G_1,G_2)\leq r(kG_1,lG_2)$$
$$r(G_1,G_2)\leq r(G_1,G_2:k) \leq r(kG_1,lG_2)$$
$$r^{*}(kG_1,lG_2)\leq r(kG_1,lG_2)$$
$$r^{*}(G_1,G_2:k)\leq r(G_1,G_2:k)$$

\end{rem}

\subsection{Non-disjoint}

We now establish the following value for exposition:

\begin{prop}
$r^{*}(K_3,K_3:2)=6$
\end{prop}

\begin{proof}
It is clear that $n=5$ is insufficient by Lemma 3.2.

Figure 1 establishes that $n=6$ is sufficient.

\begin{figure}[h]
\centering

\includegraphics[ height=4cm]{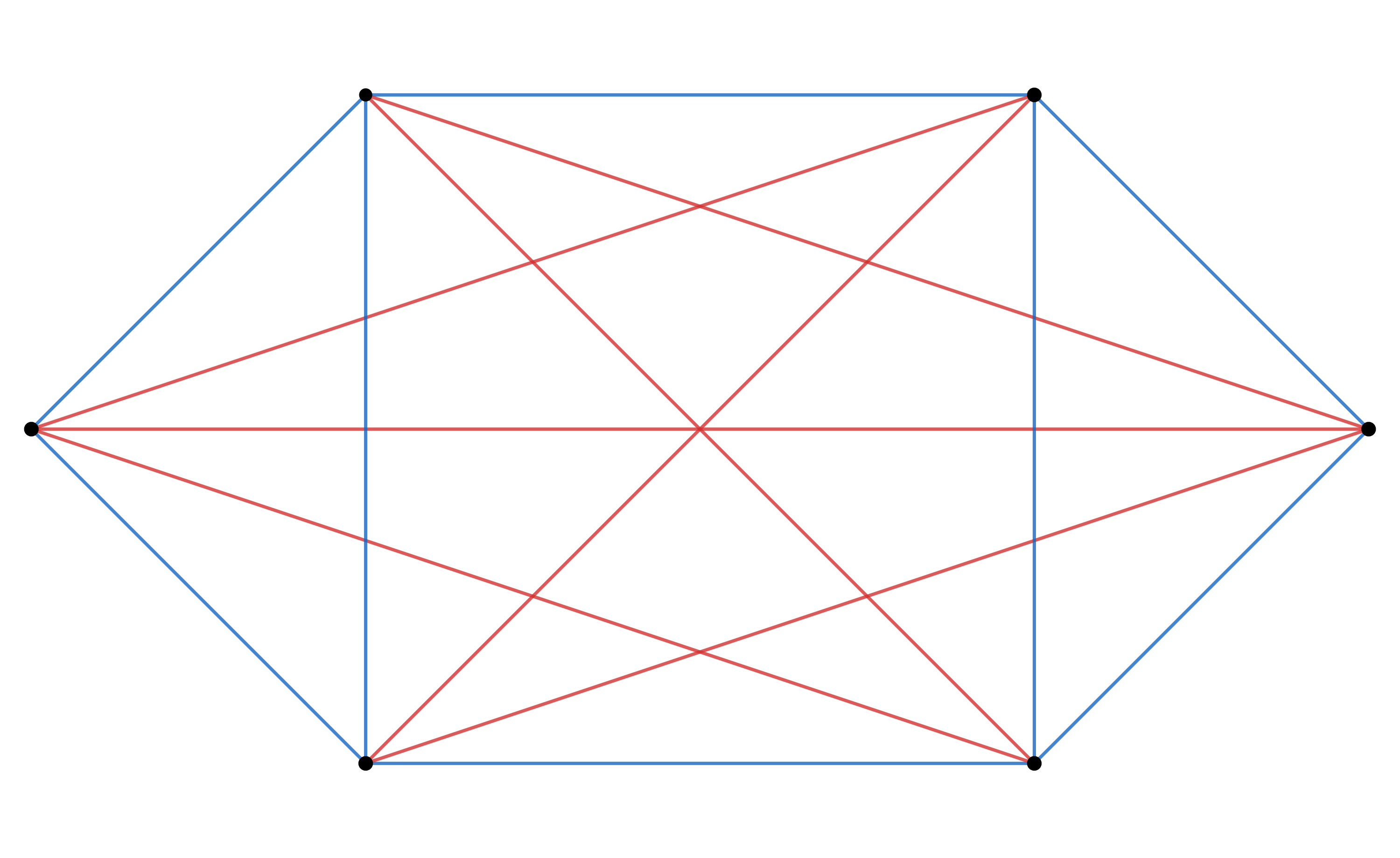}
\caption{Complete graph on six vertices with two monochromatic triangles}

\end{figure}

By the result that $r(K_3,K_3)=6$ \cite{do2019party}, we can assume a monochromatic triangle and without loss of generality blue. If this were the only triangle, removing a vertex (and all associated edges) would leave us in the situation of Lemma 3.2.

Then by Lemma 3.2 we have an all blue pentagon labelled with vertices $a,b,c,d,e$. Add a center $x$ (the vertex of the original monochromatic triangle) and two blue edges to make the known blue triangle $\{d,e,x\}$.

Then if there are no more monochromatic triangles then $ax$ and $bx$ are red. But if $ab$ is blue this gives a blue triangle $\{a,b,c\}$ and if $ab$ is red this gives a red triangle $\{a,b,x\}$. Hence, there are at least two monochromatic triangles.

\end{proof}

We now state the result of Goodman in \cite{goodman1959sets} (which generalizes the above result) and which was proved more succinctly in \cite{schwenk1972acquaintance}.

\begin{thm}
The minimum number of monochromatic triangles in a bi-colored complete graph on $n$ vertices is given by:

$${n \choose 3}-\left\lfloor  \frac{n}{2} \left\lfloor \frac{(n-1)^{2}}{4} \right\rfloor\right \rfloor$$
\end{thm}

\subsection{Disjoint}

The following is a Theorem which as far as we aware has not appeared in the literature:

\begin{thm}
$r(K_3,K_3:k)=3k+2$ for $k\geq2$.
\end{thm}

\begin{proof}
We begin by showing that $3k+1$ vertices are not enough. We construct a $2$-colored $K_{3k+1}$ by taking a blue $K_{3k-1}$, a blue $K_2$ and then joining the two by red edges.

Then necessarily there are no red monochromatic triangles and can only be $k-1$ disjoint blue triangles from the blue $K_{3k-1}$ subgraph.

We now proceed by induction on $k$ to show that $3k+2$ is sufficient.

Consider a red/blue coloring of $K_8$ and assume that it does not contain at least two disjoint monochromatic triangles. Proposition 3.5 establishes that there are at least two non-disjoint monochromatic triangles. If they are disjoint then the result follows, hence we assume they are not. Then, consider the following cases:

\underline{Case 1:} Suppose that in the $2$-coloring of $K_8$, there is a blue triangle and a red triangle, whose vertices are labeled in Figure 2.


\begin{figure}[ht]
\centering
\includegraphics[height=3cm]{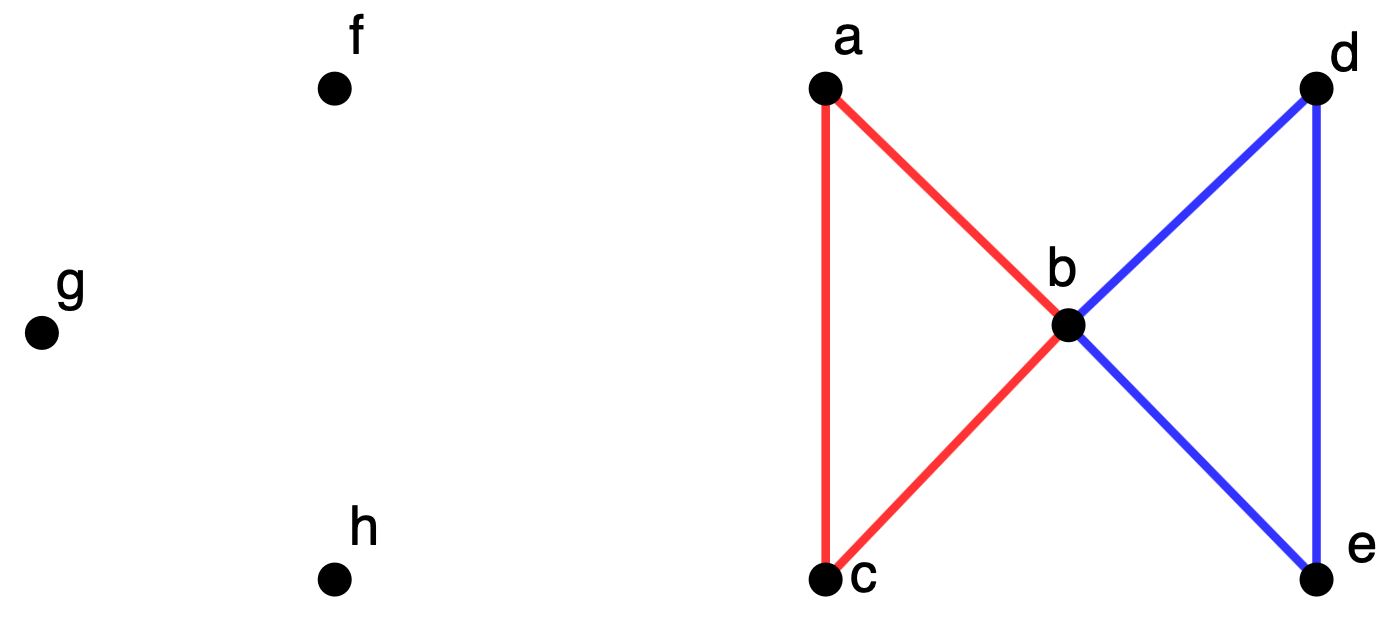}
\caption{One red triangle and one blue triangle}
\end{figure}

Now consider the subgraph induced by $\{a,c,f,g,h\}$. If it contains a monochromatic triangle, then that triangle and the triangle formed by $\{b,d,e\}$ gives at least two disjoint monochromatic triangles, which is a contradiction to our assumption. Then the subgraph generated by $\{a,b,f,g,h\}$ must be colored according to Lemma 3.2, and edge $ac$ must be included in the red $5$-cycle. It follows that the subgraph induced by $\{f,g,h\}$ must contain two red edges, say edges $fg$ and $gh$ must be red, otherwise one of the vertices in $\{e,f,g\}$ will join both $a$ and $c$ to form a monochromatic triangle that is disjoint from the blue triangle with vertices $\{b,d,e\}$.

Similarly, using a similar logic as above, if we consider the subgraph of $\{d,e,f,g,h\}$, the subgraph must be colored in accordance with Lemma 3.2, and edge $de$ must be included in the blue $5$-cycle. It follows that the subgraph induced by $\{f,g,h\}$ must contain two blue edges, say edges $fg$ and $gh$ must be blue to form a $5$-cycle. Since $\{f,g,h\}$ only has three edges, we obtain a contradiction. Hence all monochromatic triangles are the same color.

\underline{Case 2:} Suppose that in the $2$-coloring of $K_8$, there are exactly two same-colored triangles, say red, that share an edge, whose vertices are labeled in Figure 3.

\begin{figure}[ht]
\centering
\includegraphics[height=3cm]{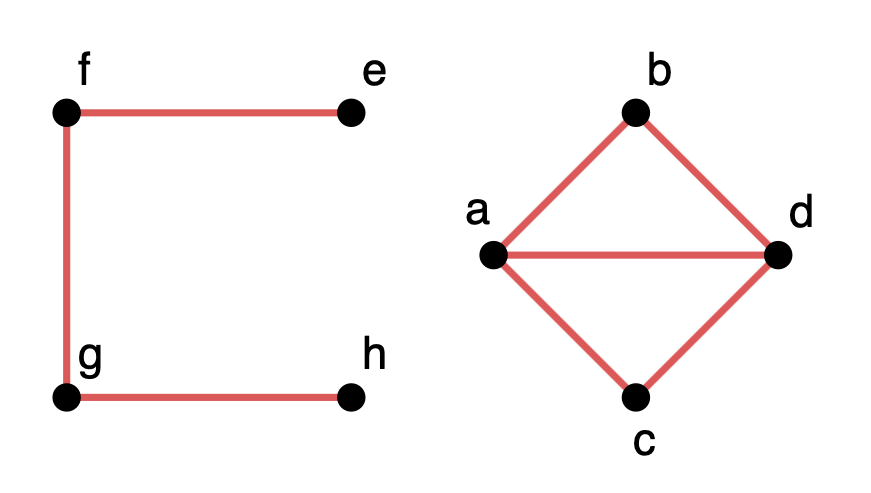}
\caption{Two red triangles sharing an edge}
\end{figure}

According to Lemma 3.2, so that the subgraph generated by $\{c,e,f,g,h\}$ does not yield a monochromatic triangle, we must have that without loss of generality $cefghc$ is a red $5$-cycle. But then to avoid a monochromatic triangle from the subgraph generated by $\{b,e,f,g,h\}$ we must have that $befghb$ is a red $5$-cycle. But then again by Lemma 3.2, edges $bg$ and $gc$ must be blue. So that we do not have a blue triangle, $bc$ must be red. Repeating this argument, we see that $aefgha$ and $defghd$ must be red $5$-cycles. Then we yield two disjoint red triangles generated by $\{b,e,d\}$ and $\{a,c,h\}$, this is in contradiction to our assumption.

\underline{Case 3:} Suppose that in the $2$-coloring of $K_8$, there are two same-colored, say red, triangles that share only one vertex, whose vertices are labeled in Figure 4.

By a similar argument to the one used twice in case 1, the subgraph generated by $\{f,g,h\}$ must contain two red edges and without loss of generality we assume that $fg$ and $gh$ are red also.

\begin{figure}[ht]
\centering
\includegraphics[height=3cm]{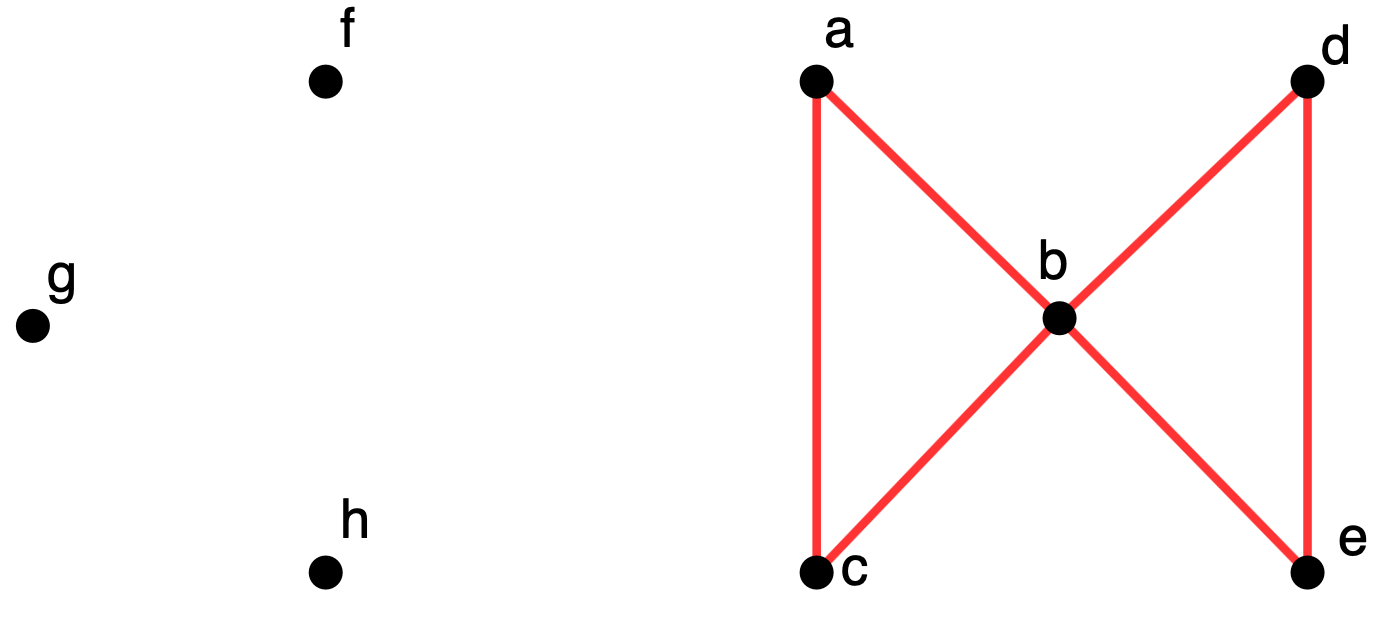}
\caption{Two red triangles}
\end{figure}

Then again without further loss of generality, we assume that $afghca$ and $dfghed$ are the red $5$-cycles. Then $ad$ must be red to avoid a blue triangle $\{a,g,d\}$ and $ce$ must be red to avoid a blue triangle $\{c,g,e\}$. But then we have red triangles $\{a,f,d\}$ and $\{c,e,h\}$ which are disjoint. Another contradiction.

Hence there are at least two disjoint monochromatic triangles. Note that three triangles cannot be disjoint by the pigeon hole principle. Hence there are exactly two disjoint monochromatic triangles. 

For the inductive step, suppose that the result is true for $k$ monochromatic triangles. So taking a $2$-colored complete graph on $3(k+1)+2=3k+5$ vertices we can assume the existence of at least one monochromatic triangle. Deleting the vertices of this triangle (and all associated edges) leaves a $2$-colored complete graph on $3k+2$ vertices which by the inductive hypothesis contains at least $k$ monochromatic triangles. Hence the $2$-colored complete graph on $3(k+1)+2$ vertices contains $k+1$ monochromatic triangles as required.

\end{proof}

\begin{cor}
The minimum number of disjoint monochromatic triangles in a 2-coloring of a complete graph on $n$ vertices is: $$\left\lfloor  \frac{n-2}{3}  \right\rfloor$$
\end{cor}

\begin{proof}
This follows as the least integer inverse of $n=3k+2$ as established by Theorem 3.7.
\end{proof}

These results  lead us into the theory of configurations by an approach that asks what the minimum number of triangles in a particular graph is as opposed to how many vertices are required to have a certain minimum number of monochromatic triangles. 

Ramsey numbers are dependent upon the fact that complete graphs form a totally ordered set. Configurations of points and lines certainly do not form a totally ordered set. However, as is recognized in \cite{graham1990ramsey}, we can always compute minimum numbers of monochromatic triangles given a particular graph. We will take this approach with some specific configurations that are symmetric with three points per line.

\section{Ramsey theory for configurations}

We begin by making the note that by  \cite{mendelsohn1987planar}, a Menger graph does not uniquely determine the configuration. Hence, Ramsey theory on configurations cannot be simply defined by Ramsey theory on the Menger graph.

As this paper deals entirely with the number of monochromatic triangles, we must define what we mean by a triangle in a configuration.

\begin{defn}

We say that the points $p_{1},p_{2},p_{3} \in \mathcal{P}$ form a triangle if there are lines $l_{1},l_{2},l_{3}$ such that:
\begin{enumerate}
    \item $p_{1},p_{2}\in l_{1}$ but $p_{3}\not\in l_{1}$
    \item $p_{2},p_{3}\in l_{2}$ but $p_{1}\not\in l_{2}$
    \item $p_{1},p_{3}\in l_{3}$ but $p_{2}\not\in l_{3}$
\end{enumerate}
\end{defn}

Note that by definition, the lines $l_{1},l_{2},l_{3}$ are unique as given any two points there is only one line that contains both. Hence, given a triangle it is equivalent to give the three points $p_{1},p_{2},p_{3}$ or the three lines $l_{1},l_{2},l_{3}$.

\begin{exmp}
For the Fano configuration, the usual Menger graph is given by figure 5.

\begin{figure}[ht]
\centering
\includegraphics[height=3cm]{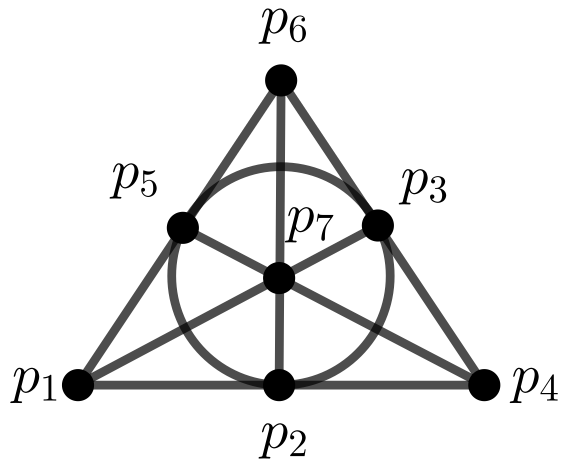}
\label{fig:young1}
\caption{Fano configuration}
\end{figure}

As a graph, the points (vertices) $p_{1},p_{4},p_{6}$ do not form a triangle, yet as a configuration they certainly do.
\end{exmp}

From this example we also see that disjointness is less clear. Triangle $p_{1},p_{2},p_{5}$ is disjoint from triangle $p_{3},p_{4},p_{7}$ as graph triangles, but $p_{2}$ and $p_{4}$ lie on the same line so can they be considered disjoint in a configuration context? To avoid this ambiguity we will stick to the non-disjoint case.

We now give the following definition:

\begin{defn}
Given a configuration, we define a line 2-coloring to be an assigning of the color red or blue to each line of the configuration.

A triangle is monochromatic if each of the lines $l_{1},l_{2},l_{3}$ are the same color.

\end{defn}

We now immediately give the following results:

\begin{thm}
The minimum number of monochromatic triangles in a line 2-coloring of the Fano configuration is 4.
\end{thm}

\begin{proof}
We give a coloring that has only 4 monochromatic triangles, three blue and one red in figure 6.

\begin{figure}[ht]
\centering
\includegraphics[height=3cm]{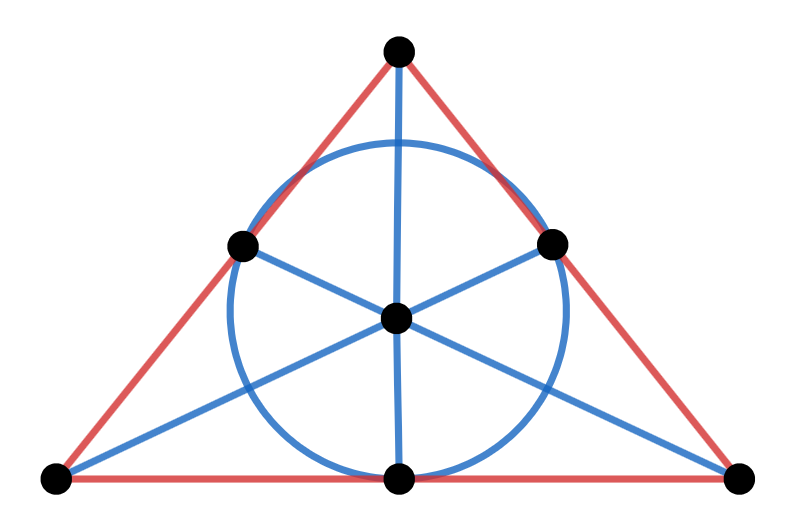}
\label{fig:young1}
\caption{Minimal Fano coloring}
\end{figure}

To see that this is minimal, suppose there are no monochromatic triangles. Without loss of generality color a triangle with 2 red lines, line $125$ and line $567$, and 1 blue line, line $147$, so that no red triangle is created. Color another line blue, line $136$. In order to create no red triangles, color another line blue. Thus, there will be at least 1 blue triangle.

So we may assume that there is at least one monochromatic triangle. Without loss of generality color a red triangle, lines $147$, $567$, and $125$. In order to create no other red triangles, all the other lines have to be colored blue. Thus, given only one monochromatic triangle of a given color, there will be 3 monochromatic triangles of the opposite color, and there are at least 4 monochromatic triangles. 

Finally, assume there are exactly two monochromatic triangles of a given color. Without loss of generality color a red triangle, lines $147$, $567$, and $125$. To create a second red triangle, color any other line red, line $136$. Note that given the symmetry of the Fano configuration, any two such red triangles must be as such. Then necessarily there are then 3 red triangles. 

Given the three red triangles, there are either more red triangles or a blue triangle formed by coloring the remaining lines blue. The result then follows. 

\end{proof}

\begin{cor}
Given a line 2-coloring of the Fano configuration with 4 monochromatic triangles, the minimum number of blue triangles is 1.
\end{cor}

\begin{thm}
The minimum number of monochromatic triangles in a line 2-coloring of the Young configuration is 0.

\end{thm}

\begin{proof}
We give the following coloring in figure 7 for proof.
\begin{figure}[ht]
\centering
\includegraphics[height=4cm]{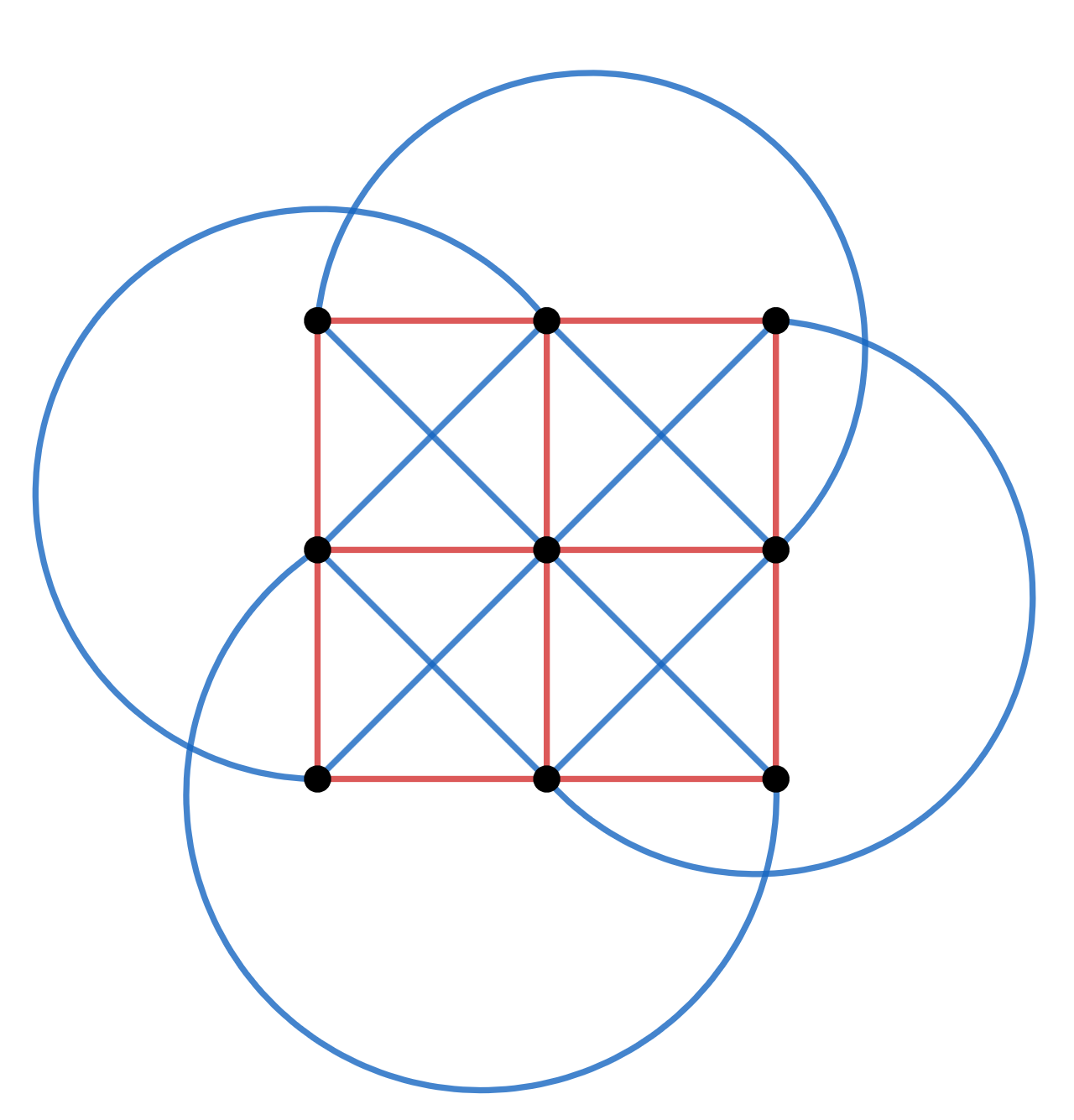}
\caption{Minimally colored Young configuration}
\end{figure}

\end{proof}

\begin{thm}
The minimum number of monochromatic triangles in a line 2-coloring of the Möbius-Kantor configuration is 0.

\end{thm}

\begin{proof}
Again we present a minimal coloring in figure 8. 
\begin{figure}[ht]
\centering
\includegraphics[height=3.5cm]{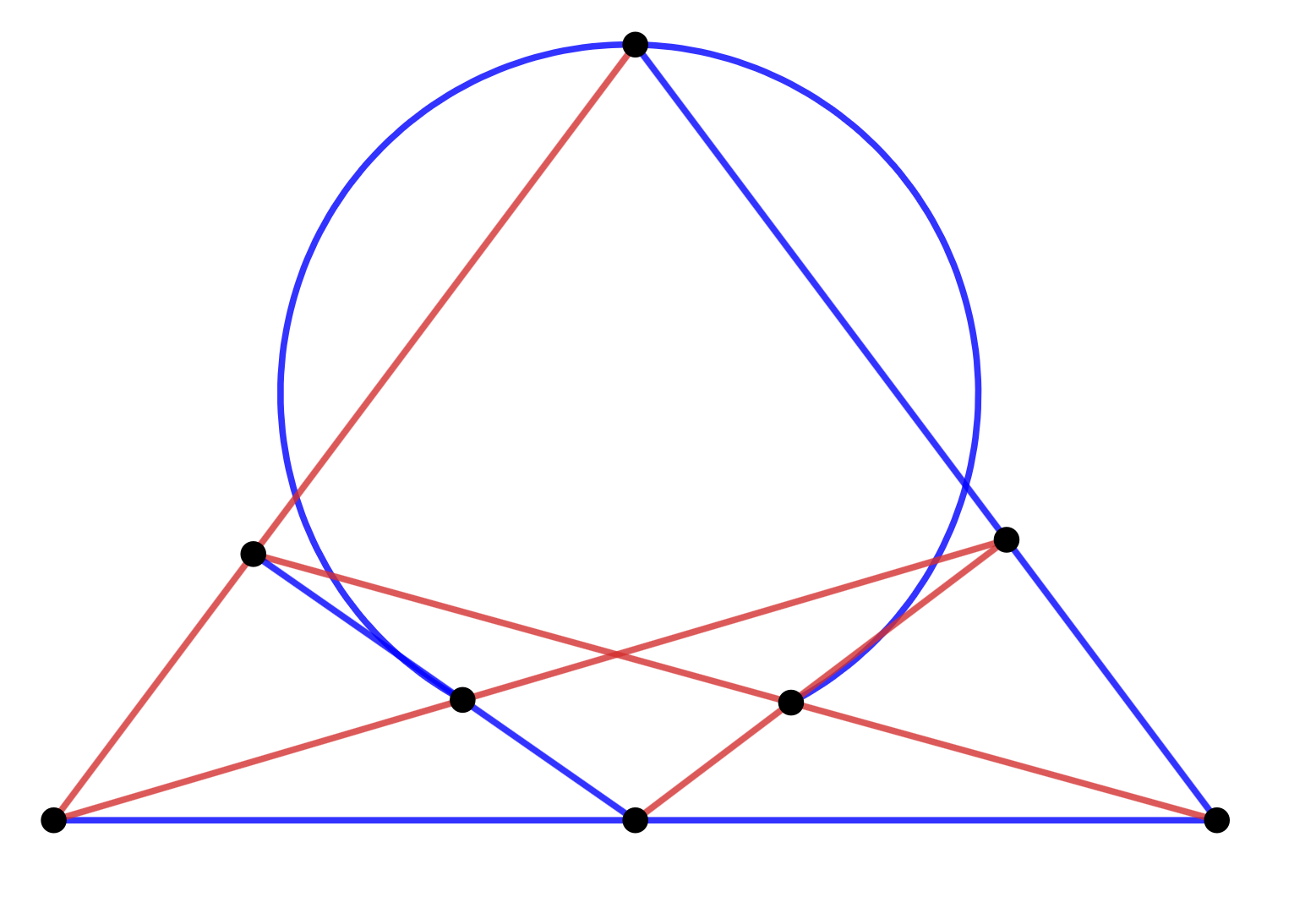}
\caption{Minally colored Möbius-Kantor configuration}
\end{figure}
\end{proof}

\begin{thm}
The minimum number of monochromatic triangles in a line 2-coloring of the Pappus configuration is 0.

\end{thm}

\begin{proof}
We give the following coloring in figure 9 to illustrate the result.

\begin{figure}[ht]
\centering
\includegraphics[height=5cm]{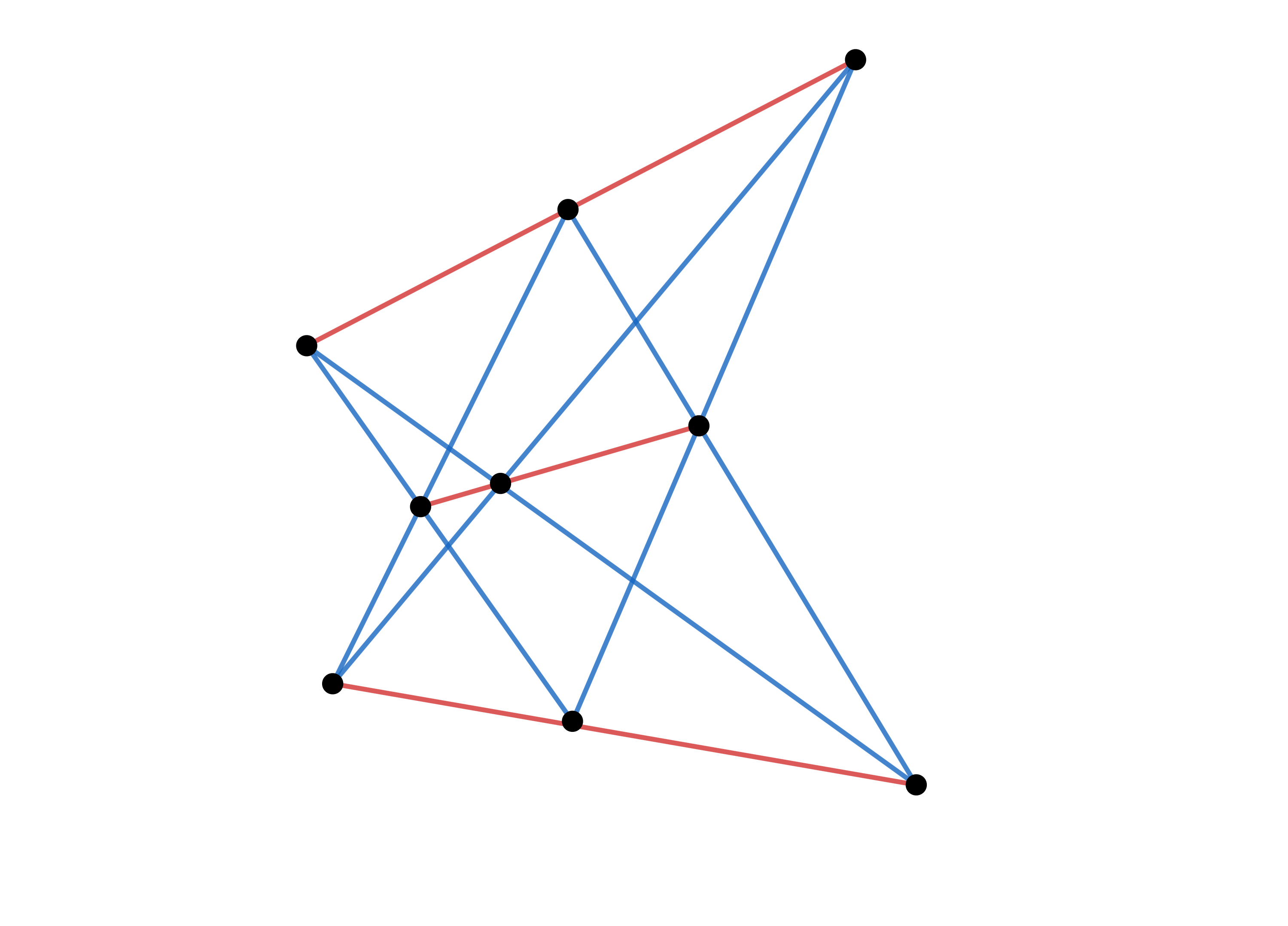}
\caption{Pappus' configuration}
\end{figure}

\end{proof}

\begin{thm}
The minimum number of monochromatic triangles in a line 2-coloring of the Desargues configuration is 0.

\end{thm}

\begin{proof}
We give here in figure 10 the coloring that achieves zero monochromatic colorings. Note how the configuration is represented as in Desargues' theorem. See \cite{dembowski1997finite} for more details.

\begin{figure}[ht]
\centering
\includegraphics[height=4cm]{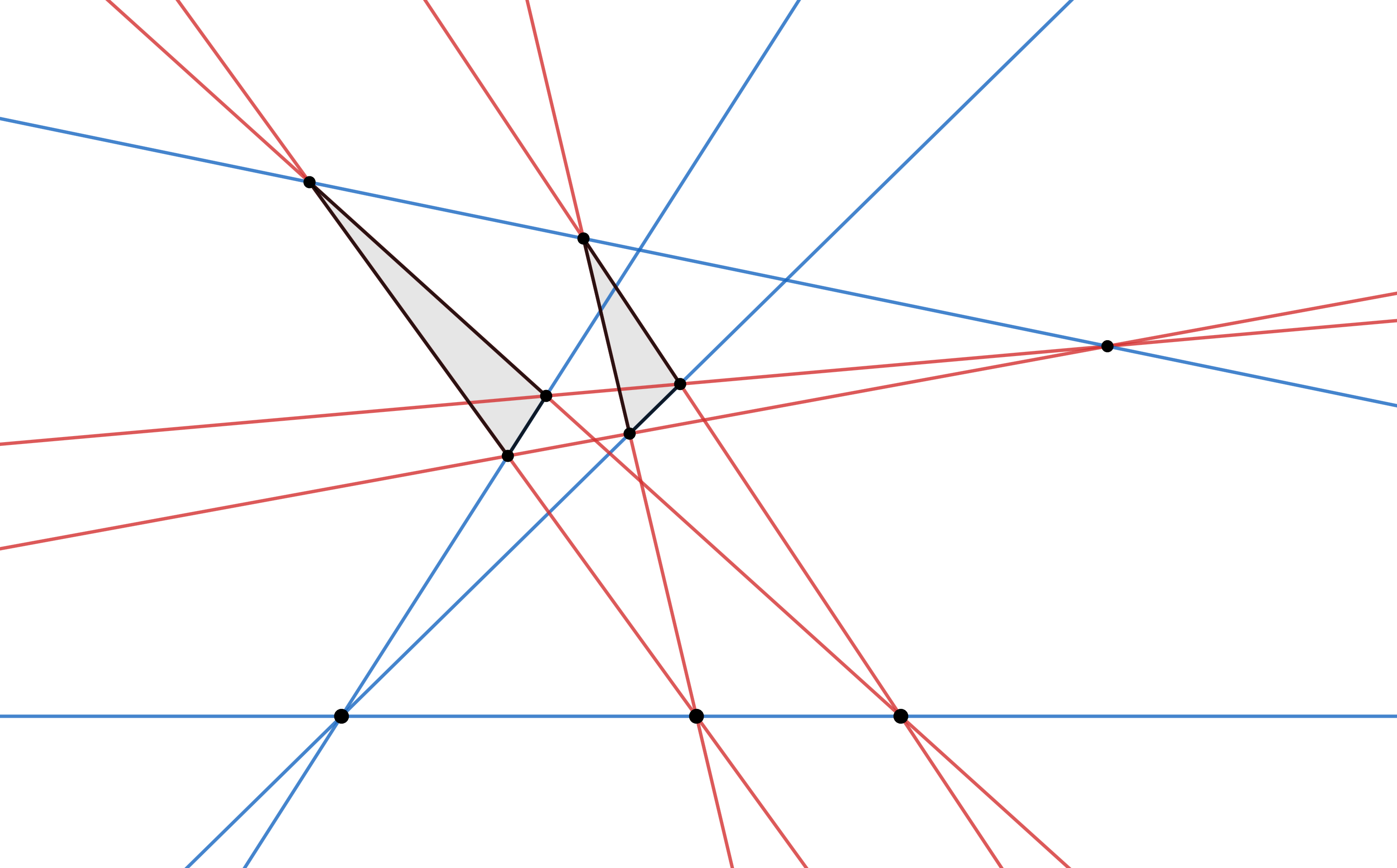}
\caption{Desargues' configuration}
\end{figure}

\end{proof}

The code written in the GAP coding language in the github repository at \cite{GAPcode} can both verify the above results and compute the minimum number of monochromatic triangles given a line 2-coloring of a configuration.

We make the remark here that if a configuration is triangle-free, then trivially the minimum number of monochromatic triangles will of course be zero.

\section{Incidence switches}

In \cite{bras2012semigroup} and \cite{pisanski2012configurations}, the concept of the addition of configurations or an incidence switch is given. We define this here as a connected sum (using the terminology of algebraic topology) in order to provide a constructive result for computing the minimum number of monochromatic triangles.

\begin{defn}
Given two $(n_{s},m_{k})$-configurations we define a connected sum of the two configurations to be a $(2n_{s},2m_{k})$-configuration constructed as follows: Take points $p_{1}$ and $p_{2}$ from each configuration and lines $l_{1}$ and $l_{2}$ that each are on. Then replace $p_{1}$ with $p_{2}$ in $l_{1}$ and similarly $p_{2}$ with $p_{1}$ in $l_{2}$. 
\end{defn}

Visually, we can see this connected sum in figure 11 as:

\begin{figure}[ht]
\centering
\includegraphics[height=3cm]{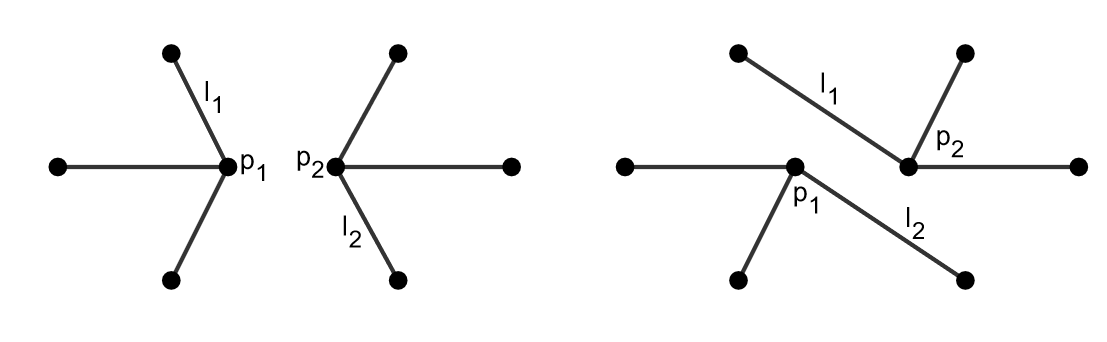}
\caption{Connected sum}
\end{figure}

It is clear that this connected sum is still a configuration and indeed a $(2n_{s},2m_{k})$-configuration. 

Recall that a configuration is \textit{flag-transitive} if for any two incidences of points and lines (flags) there exists a line-preserving bijection of the point set that sends the first incidence to the second. It then follows that these connected sums are not necessarily unique if the configurations are not flag-transitive - dependent on the choices of points and line.

We now give the following proposition:

\begin{prop}
Given a connected sum of two $(n_{s},m_{k})$-configurations, the point set can be partitioned into points from the first configuration and points from the second. Given this partition there is no triangle that has points from both partitioning sets. 
\end{prop}

\begin{proof}
This follows as the connected sum has only two lines that contain points from both the original configurations. Three lines would be required to form a triangle.
\end{proof}

This then immediately leads to the following corollaries:

\begin{cor}
Given a connected sum of two $(n_{s},m_{k})$-configurations the number of triangles is less than or equal to the sum of the number of triangles of the two configurations.
\end{cor}

\begin{cor}
Given a connected sum of two $(n_{s},m_{k})$-configurations and a line 2-coloring, the minimum number of monochromatic triangles is less than or equal to the sum of the minimum number of monochromatic triangles given line 2-colorings of the two configurations.
\end{cor}

\section{Maximal sets of mutually intersecting lines}

We here make the observation that the minimum number of monochromatic triangles was only non-zero for the Fano plane (for the configurations that we dealt with in section 4). We also note that the Fano plane consists of seven mutually intersecting lines, whereas for example Pappus's configuration contains maximally three intersecting lines and the Mobius-Kantor four. We take this observation to state formally the result:

\begin{thm}
Suppose a symmetric $v_{3}$ configuration contains a maximal set of mutually intersecting lines of size $i$. Then $i\geq6$ implies that the minimum number of a monochromatic $2$-colorings is non-zero. Furthermore, given such a subset of six mutually intersecting lines, the minimum number of monochromatic triangles given a line $2$-coloring of the set is two.
\end{thm}

We note here that the maximum size of a set of mutually intersecting lines is seven (given a line, each of the three points on the line can be on two more lines) and this refers exactly to the Fano configuration where we have seen the minimum number of monochromatic triangles given a line $2$-coloring is four.

\begin{proof}
Take a set of six mutually intersecting lines. It is clear that the only possible arrangement of six such lines is as in figure 12

\begin{figure}[ht]
\centering
\includegraphics[height=3cm]{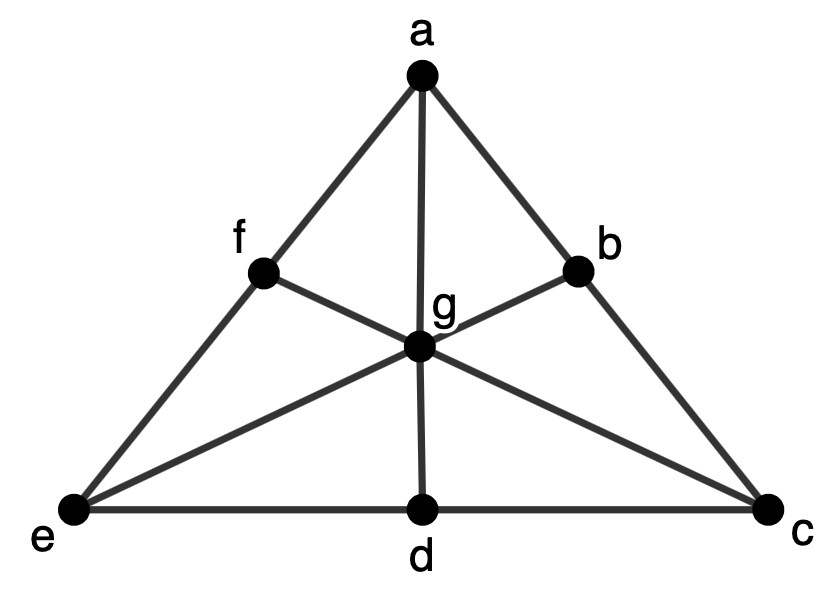}
\caption{Six mutually intersecting lines}
\end{figure}

It then follows that without loss of generality if lines $abc$ and $cde$ are blue and $aef$ is red, then $bge$ must be red, but in that case $fgc$ and $agd$ both form monochromatic triangles whether red or blue. Minimally two are formed if $fgc$ and $agd$ are different colors as seen in figure 13.

\begin{figure}[ht]
\centering
\includegraphics[height=3cm]{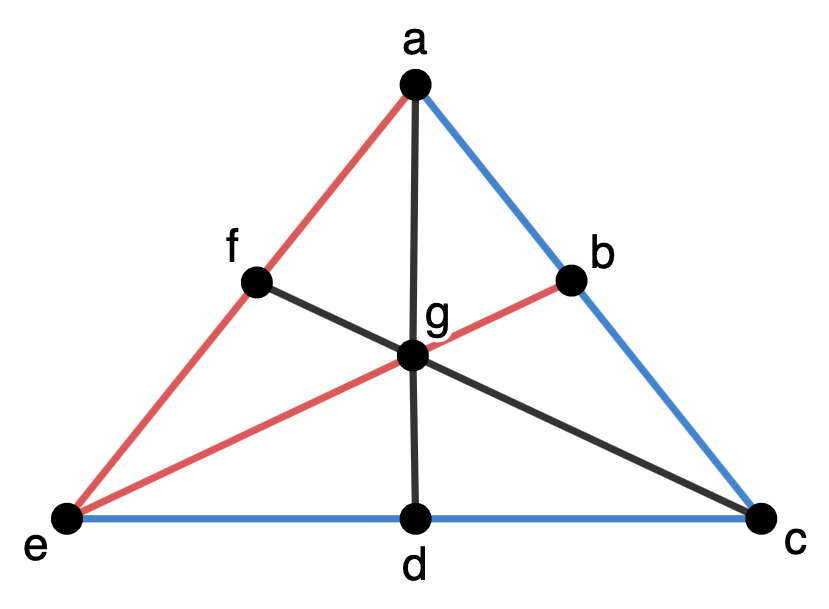}
\caption{Six mutually intersecting lines with necessary coloring}
\end{figure}

This yields the required two monochromatic triangles.
\end{proof}

The immediate corollary is the following.

\begin{cor}
If a symmetric $v_{3}$ configuration has $k$ disjoint sets of $6$ mutually intersecting lines then the minimum number of monochromatic triangles given a line $2$-coloring will be greater than or equal to $2k$.
\end{cor}

We then see that the connected sum of configurations gives an excellent example of this corollary:

\begin{exmp}
Take a connected sum of two Fano configurations. The Fano configuration is flag-transitive so this is unique. It is then clear that there are two disjoint sets of six mutually intersecting lines - one each from the original configurations. By the corollary, the minimum number of monochromatic triangles given a line $2$-coloring should be greater than or equal to $4$. Our GAP code confirms this and establishes that it is exactly $4$.
\end{exmp}

We here also give a couple of particular examples of configurations with maximally 5 mutually intersecting lines. The first has minimally no monochromatic triangles given a line $2$-coloring and the second has minimally one monochromatic triangle given a line $2$-coloring.

\begin{exmp}
\begin{verbatim}
    
{{1,2,3},{3,4,5},{1,5,6},{3,6,7},{2,5,7},{1,8,9},{6,9,10},{2,11,13},{7,12,13},{4,8,10},
{4,11,12},{9,13,14},{8,11,14},{10,12,14}}
\end{verbatim}

\end{exmp}

\begin{exmp}
\begin{verbatim}

{{1,2,3},{1,4,5},{3,5,6},{1,6,7},{3,4,7},{4,6,8},{1,10,15},{2,13,14},{5,9,11},{7,12,16},
{8,9,10},{11,12,13},{14,15,16},{8,11,14},{9,12,15},{10,13,16}}
\end{verbatim}
\end{exmp}

This second example makes it clear that the converse of Theorem 6.1 is not true. We conjecture that the following partial converse is true:

\begin{conj}
Suppose a symmetric $v_{3}$ configuration contains a maximal set of mutually intersecting lines of size $i$. If minimum number of monochromatic triangles in a line $2$-colorings is non-zero, then $i>4$.
\end{conj}

\bibliographystyle{plain}  
\bibliography{references} 

\end{document}